\documentclass[runningheads]{llncs}
\usepackage{physics}

  \newenvironment{bsmallmatrix}
  {\left[\begin{smallmatrix}}
  {\end{smallmatrix}\right]}
\usepackage{enumitem} 
\usepackage{bm}
\usepackage{stmaryrd}
\usepackage[T1]{fontenc}
\usepackage{lscape}
\usepackage{graphicx}
\usepackage{hyperref}
\hypersetup{
   colorlinks   =  true,
    linkcolor    = blue,
    citecolor    = red,
     urlcolor	=blue,     
	urlbordercolor = {1 1 0}
}
\usepackage{color}

\usepackage{amsmath, amsfonts, amssymb}
\begin{document}
\title{Novel pathways in $k$-contact geometry}
\author{Tomasz Sobczak\inst{1}\orcidID{0009-0002-9577-0456}, Tymon Frelik\orcidID{0009-0008-2996-4571}}

% \and
\authorrunning{Tomasz Sobczak, Tymon Frelik}
\institute{Department of Mathematical Methods in Physics,\\ University of Warsaw, Warsaw, Poland}

\maketitle 
\begin{abstract}
Our study of Goursat distributions originates new types of $k$-contact distributions and Lie systems with applications. In particular, families of generators for Goursat distributions on $\mathbb{R}^4, \mathbb{R}^5$ and $\mathbb{R}^6$ give rise to Lie systems and we characterise Goursat structures that are $k$-contact distributions. Our results are used to study the zero-trailer and other systems via Lie systems and $k$-contact manifolds. New ideas for the development of superposition rules via geometric structures and the characterisation of $k$-contact distributions are given and applied. Some relations of $k$-contact geometry with parabolic Cartan geometries are inspected. 
\end{abstract}
\section{Introduction}

Recently, $k$-contact geometry appeared as a generalisation of contact geometry to analyse field theories \cite{GGMRR_20,GGMRR_21}. Recent research \cite{LRS_24} proposed that $k$-contact forms, a central element in $k$-contact geometry, are more commonly applied in practice than being an essential characteristic: $k$-contact distributions are more important. Applications to Goursat distributions appeared in \cite{LRS_24}. 

This paper studies Goursat distributions on manifolds of dimension four to six that are $k$-contact distributions. A technique to characterise $k$-contact distributions is provided, which is better than previous methods \cite{LRS_24}, and applied in our characterisation. We show that not every Goursat distribution is a $k$-contact distribution. Nevertheless, all classes on $\mathbb{R}^4$, $\mathbb{R}^5$, and $\mathbb{R}^6$ admit bases of generating vector fields included in a finite-dimensional Lie algebra of vector fields: a Vessiot--Guldberg (VG) Lie algebra \cite{CL_11}. 

{\it Lie systems} are $t$-dependent  systems of ordinary differential equations whose general solution can be written as a $t$-independent function, called {\it superposition rule}, of a generic family of particular solutions and a set of constants related to initial conditions. The Lie--Scheffers theorem states that every Lie system is associated with a $t$-dependent vector field given by a linear combination with $t$-dependent coefficients of vector fields spanning a finite-dimensional Lie algebra: a  {\it Vessiot--Guldberg (VG) Lie algebra}. Although Lie systems are quite restrictive among differential equations, they admit relevant applications in physics and  mathematical properties \cite{CGM_00,CL_11,LS_20}.
Lie systems on the real line and the plane were fully classified by Lie \cite{LS_20}. They are quite well understood, and their properties and superposition rules have been studied in most cases \cite{BBHLS_15}. Lie systems on manifolds of dimension larger than two are scarcely known, apart from particular cases \cite{CL_11,LS_20,Ram_02}.

Overall, this paper serves as a brief summary of emerging research avenues in $k$-contact geometry, which will be explored in future studies.
 We have shown that Goursat distributions on $\mathbb{R}^4$, $\mathbb{R}^5$, and $\mathbb{R}^6$ define new examples of locally automorphic Lie systems. This was unexpected and relevant, as they provide Lie systems on manifolds of dimension larger than two with new applications to trailer systems. Previous Lie systems  in control theory are recovered \cite{Ram_02}. Table \ref{tab:Lie_systems} summarises our classification results on Goursat distributions with their associated Lie systems and $k$-contact  distributions (if available). %Our classification of Goursat distributions expands the potential applications of Lie systems on manifolds of dimension larger than two and analyses the applications of $k$-contact geometry to dynamical systems. In particular, applications of our results to zero- and one-trailer systems \cite{PR_01} are developed. 
Finally, the relation of $k$-contact distributions and certain results on Cartan geometry is studied. We inspect three examples of distributions on $\mathbb{R}^5$, $\mathbb{R}^6$, and $\mathbb{R}^7$, which define certain flat parabolic Cartan geometries and show that they are $k$-contact. From the example of $(2,3,5)$ distributions we see that flatness of a Cartan geometry may not be the necessary condition for the corresponding distribution to be $k$-contact, which motivates further research.

\subsection{$k$-contact geometry}
Let us briefly introduce $k$-contact geometry (see \cite{GGMRR_20,LRS_24} for details).  Hereafter, $M$ is assumed to be a connected manifold and structures are smooth unless otherwise stated. Moreover, $\mathfrak{X}(M)$ is the space of vector fields on $M$.

% \begin{definition}\label{dfn:k-contact-manifold}
%     A \textit{$k$-contact form on an open $U\subset M$} is a differential one-form on $U$ taking values in $\mathbb{R}^k$, let us say $\bm\eta\in\Omega^1(U,\mathbb{R}^k)$, such that:
%         i) $0\neq \ker \bm\eta\subset\text{T} U$ is a regular distribution of corank $k$, ii) $\ker \text{d}\bm\eta\subset\text{T} U$ is a regular distribution of rank $k$, iii) $\ker \bm\eta\cap\ker \text{d}\bm\eta  = 0$. If $\bm \eta$ is defined on the whole $M$, the pair $(M,\bm\eta)$ is a \textit{co-oriented $k$-contact manifold} and $\ker\text{d}\bm\eta$ is its  {\it Reeb distribution}. 
%     \end{definition}

    $k$-Contact distributions, defined in Definition \ref{Def:LocalEqu}, are a recently proposed class of distributions which generalizes $k$-contact geometry. Recall that a regular distribution $\mathcal{D}$ is {\it maximally non-integrable} if $\rho:\mathcal{D}\times \mathcal{D}\rightarrow TM/\mathcal{D}$ defined as $\rho(X,Y)=\pi\left([X,Y]\right)$ for every $X,Y\in \mathcal{D}$ and the natural vector bundle projection $\pi:\text{T}M \rightarrow \text{T}M/\mathcal{D}$, is non-degenerate. A {\it Lie symmetry of $\mathcal{D}$} is a vector field on $M$ whose Lie brackets with vector fields taking values in $\mathcal{D}$ give vector fields taking values in $\mathcal{D}$. Finally, an  {\it integrable $k$-vector field} on $M$ is a family of $k$ commuting vector fields on $M$.  
    \begin{definition}\label{Def:LocalEqu} 
    A $k$-contact distribution $\mathcal{D}$ on $M$ is a maximally non-integrable distribution  if for every open subset $U \subset M$ it admits an integrable $k$-vector field ${\bf S} = (S_1,\ldots, S_k)$ of Lie symmetries of $\mathcal{D}|_U$, such that
    \begin{equation*}\label{eq:Dec}
        \langle S_1,\ldots,S_k\rangle \oplus \mathcal{D}|_U = \text{T} \,U\,.
    \end{equation*}
    If $S_1,\ldots,S_k$ are globally defined, they are called {\it Reeb vector fields}.
\end{definition} 
Given $ X_1,X_2\in \mathfrak{X}(M)$ with $X_1\wedge X_2\neq 0$, we define ${\rm ad}^{k+1}_{X_2}X_1=[X_2,{\rm ad}^kX_1]$ for $k\in \mathbb{N}$ and the distributions
$$
\mathcal{D}_{X_1,X_2}^k=\langle X_1,X_2,{\rm ad}_{X_2}X_1,\ldots,{\rm ad}_{X_2}^kX_1\rangle.
$$
These distributions are not regular in general but, under quite practical conditions, are invariant relative to the Lie symmetries of $\mathcal{D}=\langle X_1,X_2\rangle$. This will be used here to state when $\mathcal{D}$ is not a $k$-contact distribution.

\subsection{Lie systems}
Let us describe some notions related to Lie systems \cite{CL_11,LS_20}. A \textit{t-dependent vector field} is a map $X : \mathbb{R} \times M \ni (t,x) \mapsto X(t,x) \in \text{T} M$ such that $\pi \circ X= \pi_{2}$ for projections $\pi_2:(t,x)\in \mathbb{R}\times M\mapsto x\in M$ and $\pi \colon \text{T}M \rightarrow M $.
Then, each $t$-dependent vector field $X$ on $M$ is equivalent to a family $\{X_{t}\}_{t \in \mathbb{R}}\subset \mathfrak{X}(M)$.  

 The \textit{system associated} with $X$ is the system of first-order ordinary differential equations
\begin{equation}
\label{eq:integral_curve}
    \tfrac{\text{d} x}{\text{d} t} = X(t,x), \qquad x\in M,\qquad t\in \mathbb{R},
\end{equation}
whose solution $x : \mathbb{R} \rightarrow M$ with $x(0)=x_0$ is called the \textit{integral curve} of $X$ with initial condition $x_0$ at $t=0$. Every $t$-dependent vector field has an associated system of ODEs \eqref{eq:integral_curve} and vice versa. 

    A system $X$ on $M$ admits a \textit{superposition rule} if there exists a $t$-independent map $\Phi : M^{\ell} \times M \rightarrow M$ of the form $\label{eq7}
    x=\Phi(x_{(1)},\ldots,x_{(\ell)};k)\,,
$ such that every generic solution of $X$ can be retrieved as 
$ x(t)=\Phi(x_{(1)}(t),\ldots,x_{(\ell)}(t);k)\,,
$ where $x_{(1)}(t),\ldots,x_{(\ell)}(t)$ is any generic family of particular solutions of   \eqref{eq:integral_curve} and $k \in M$.
%\qeddiamond\end{definition}
 In general, superposition rules are defined only on a open subset of $M^\ell\times M$, but we will hereafter skip this technical detail. Moreover, the last copy of $M$ in $M^\ell\times M$ is aimed to parametrise the generated solutions, which are related to the initial conditions of \eqref{eq:integral_curve}.
 
 Lie systems are characterised as follows \cite{Lie1}.

\begin{theorem}[Lie theorem]
A system $X$ on $M$  admits a superposition rule if and only if $
X(t,x)=\sum_{\alpha=1}^rb_{\alpha}(t)X_{\alpha}(x), 
$    
with $X_{1},\ldots,X_{r}$ spanning an $r$-dimensional VG Lie algebra $V$ on $M$.
\end{theorem}
We will consider \textit{locally automorphic Lie systems}, namely a Lie system $X$ on $M$ such that $\dim V=\dim M$ and $V$ spans $TM$ (see \cite{Gracia2019} for details). This allows us to determine a local diffeomorphism mapping $V$ onto  the Lie algebra of right-invariant vector fields on a Lie group $G$. This relation allows one to derive their superposition rules and other properties in a simple manner.

\section{Lie systems, $k$-contact geometry, and Goursat geometry}

For the sake of clarity, we provide the definition of the Goursat distribution \cite{PR_01}. Recall that the \textit{derived flag} of distribution $\mathcal{D}$ on an $m$-dimensional manifold $M$ is a sequence defined as $\mathcal{D}^{(0)}=\mathcal{D}$ and $\mathcal{D}^{(\ell+1)}=\mathcal{D}^{(\ell)}+[\mathcal{D}^{(\ell)},\mathcal{D}^{(\ell)}]$ for $\ell \geq0$, where $[\mathcal{D}^{(\ell)}, \mathcal{D}^{(\ell)}]$ denotes the distribution spanned by Lie brackets $[X,Y]$ of all vector fields $X, Y \in \mathcal{D}^{(\ell)}$. A \textit{Goursat structure} on $M$ with $m\geq3$ is a rank two distribution $\mathcal{D_{G}}$, such that, for every point $x \in M$, elements of its derived flag satisfy $\dim\mathcal{D_G}^{(\ell)}=\ell+2$ for $\ell=0,\ldots,m-2$.

Let us show that Goursat distributions on $ \mathbb{R}^4,\mathbb{R}^5$ and $\mathbb{R}^6$ 
 give rise to new examples of Lie systems. The following theorem follows from the classification of Goursat distributions on $ \mathbb{R}^4,\mathbb{R}^5$ and $\mathbb{R}^6$ by the Kumpera-Ruiz normal form (see \cite{KR82,PR_01} for details) and simple calculations that follow from the information in Table \ref{tab:Lie_systems}. 

\begin{landscape}
\begin{table}[]
    \centering
    \scalebox{0.78}{\begin{tabular}{|c|p{1.3cm}|p{6.9cm}|p{4.4cm}|p{7.5cm}|p{2.9cm}|}
      \hline
      Class& Manifold&VG Lie algebra & Structure constants& Reeb vector fields $R$/commuting Lie symmetries $S$ &$k$-contact\\
        \hline
      $1$&  \centering$\mathbb{R}^4$&$X_{1}=\partial_4, \newline X_{2}=x_4\partial_3+x_3 \partial_2+\partial_1, X_3=\partial_3, X_4=\partial_2$&$c_{123}=1, c_{234}=-1$&$R_1=\partial_2, \newline R_2=x_1\partial_2+\partial_3$&two-contact\\
      \hline
      $2$   &   \centering$\mathbb{R}^5$&$X_1=\partial_5,  X_2=x_5\partial_4+x_{4}\partial_3+x_{3}\partial_2+\partial_1, \newline X_{3}=\partial_{4}, \quad 
      X_{4}=\partial_{3},  \quad X_{5}=\partial_{2}$&$c_{123}=1,  
      \newline c_{234}=c_{245}=-1$&$R_1=\partial_2, R_2=x_1\partial_2+\partial_3,\newline R_3=\tfrac{x_1^2}{2}\partial_2+x_1\partial_3+\partial_4$&three-contact\\
      \hline
       $3$  &\centering$\mathbb{R}^5$&$X_{1}=\partial _5,  X_{2}=x_5(\partial _1+x_{3}\partial _2+x_4\partial _3)+\partial_4, \newline X_{3}=\partial_1+x_{3}\partial_2+x_{4}\partial_3, \quad X_{4}=\partial_{3}, \newline X_{5}= \partial_2, \quad X_{6}=x_{5}\partial_2$ 
       &$c_{123}=c_{165}=c_{234}=1, \newline 
       c_{246}=c_{345}=-1$&$\text{for}\; x_{5} \neq 0 \newline
       S_1=\partial_2,S_2=x_1\partial_2+\partial_3, S_3=\tfrac{x_1^2}{2}\partial_2+x_1\partial_3+\partial_4\newline
       \text{for}\; x_{5}=0 \newline
        \;S_{1}=\partial_{1}, \;S_{2}=\partial_{2}, \;S_{3}=\partial_{3}, \;
       $&three-contact\\
       \hline
       $4$ &\centering$\mathbb{R}^6$&$X_{1}=\partial_6,  
       X_{2}=x_{6}\partial_5+x_{5}\partial_4+x_{4}\partial_{3} +x_{3}\partial _{2}+\partial _{1},\newline   X_{3}=\partial_5,\;  X_{4}= \partial_4,\; X_{5}=\partial_{3},\; X_{6}=\partial_{2}$ 
       &$c_{123}=1, \newline c_{234}=c_{245}=c_{256}=-1$&$R_1=\partial_2, R_2=x_1\partial_2+\partial_3,\newline R_3=\tfrac{x_1^2}{2}\partial_2+x_1\partial_3+\partial_4, R_4=\tfrac{x_1^3}{6}\partial_2+\tfrac{x_1^2}{2}\partial_3+x_1\partial_4+\partial_5$&four-contact\\
       \hline
       $5$&\centering$\mathbb{R}^6$&$X_{1}=\partial _{6}, X_{2}=\partial _5+x_6(x_5\partial _{4}+x_{4}\partial _{3}+x_{3} \partial _{2}+\partial _{1}),
       \newline
       X_{3}=\partial_1+x_{3}\partial_2+x_{4}\partial_3+x_{5}\partial_4, X_{4}=\partial_4, \newline X_{5}=x_{6}\partial_3,X_6=\partial_3,X_7=(x_6)^2\partial_2,
       \newline
       X_{8}=x_{6}\partial _2,X_{9}=\partial_2$&$c_{123}=c_{189}=c_{234}=c_{156}=1, \newline
       c_{248}=c_{257}=c_{268}=c_{346}=c_{358}=c_{369}=-1, \newline 
       c_{178}=2$&$R_1=\partial_2,R_2=x_1\partial_2+\partial_3, \newline R_3=\tfrac{x_1^2}{2}\partial_2+x_1\partial_3+\partial_4,R_4=\tfrac{x_1^3}{6}\partial_2+\tfrac{x_1^2}{2}\partial_3+x_1\partial_4+\partial_5$&four-contact\\
       \hline
       6&\centering$\mathbb{R}^6$&$X_{1}=\partial _6, X_{2}=x_{6}\partial _{5}+\partial _{4}+x_{5}\left(x_{4}\partial _{3}+x_{3}\partial _{2}+\partial _{1}\right),
       \newline
       X_{3}=\partial_5, X_{4}=\partial_1+x_{3}\partial_2+x_{4}\partial_3, \newline X_{5}=\partial_3,X_{6}=x_{5}\partial_2, X_{7}=x_{6}\partial_2, X_{8}=\partial_2$&$c_{123}=c_{178}=c_{245}=c_{267}=c_{368}=1, \newline  c_{234}= c_{256}=c_{458}=-1$&&No $k$-contact\\
       \hline
       7& \centering$ \mathbb{R}^6$ & $X_{1}=\partial_{6},$ \newline $X_{2}=(x_{6}+1)\partial_{5}+\partial_{4}+x_{5}(x_4\partial_{3}+x_{3} \partial_{2}+ \partial_{1}),$\newline $X_3=\partial_5,X_4=\partial_1+x_3\partial_2+x_4\partial_3,\newline X_5=\partial_3,X_6=x_5\partial_2,X_7=\partial_2,\newline X_8=(1+x_6)\partial_2$ &$c_{123}=c_{187}=c_{245}=c_{268}=c_{367}=1, \newline c_{234}=c_{256}=c_{457}=-1$&
       
$S_1 = Y_1-18x_5D\,\partial_6\newline  S_2=Y_2+\left(12x_4 x_5^3 - 12x_1 x_5^2 - 9x_1^2 D + 36x_3 x_5 D\right)\partial_6\newline S_3=Y_3+\Big(12x_1^2(x_3 D - x_4 x_5^2) - 2x_1^3(3x_5 + 5x_4 D) - 6x_5\big(8x_3 x_4 x_5^2 + 12x_3^2 D + 9x_2 x_4 D\big) + 6x_1\big(3x_4^2 x_5^3 + 4x_2 D + x_3 x_5(8x_5 + 9x_4 D)\big)\Big)\partial_6\newline S_4=Y_4+\Big(9x_2^2 D + 48x_3^2 x_5(x_3 + x_4 x_5^2 + x_3 x_6)  - 6x_1 x_3 x_5(8x_3 x_5 + 9x_4^2 x_5^2 + 18x_3 x_4 D)  - 2x_1^3\big(3x_3 x_5 + 5x_3 x_4 D + 3x_4^2 x_5(x_4 + 2x_5 + x_4 x_6)\big)  + 3x_1^2\big(4x_4^3 x_5^3 + 7x_3^2 D + 2x_3 x_4 x_5(10x_5 + 9x_4 D)\big)  + 6x_2\big(3x_4^2 x_5^3 + 18x_3 x_4 x_5 D + x_1^2(3x_5 + 5x_4 D) - x_1(11x_3 D + 6x_4 x_5^2 + 9x_4^2 x_5 D)\big)\Big)\partial_6$

       & \parbox{2.85cm}{\centering four-contact\\(on a dense subset)} \\
       \hline
       $8$  &\centering $\mathbb{R}^6$& $X_{1}=\partial _{6},$ \newline  $X_{2}=\partial_{5}+x_{6}(\partial_{4}+x_{5}(x_{4}\partial_{3}+x_{3}\partial_{2}+\partial_{1})),$\newline
       $X_{3}=x_5\left(\partial_1+x_3\partial_2+x_4\partial_3\right)+\partial_4, \newline X_4=\partial_1+x_3\partial_2+x_4\partial_3,X_5=x_6\partial_3, \newline X_6=\partial_3, X_7=x_5(x_6)^2\partial_2,X_8=x_5x_6\partial_2, \newline X_9=x_5\partial_2,X_{10}=\partial_2,X_{11}=(x_6)^2\partial_2,\newline X_{12}=x_6\partial_2$&$c_{123}=c_{156}=c_{189}=c_{1(12)(10)}=c_{234}=c_{245}=c_{27(11)}=c_{28(12)}=c_{29(10)}=c_{346}=1, 
       \newline c_{257}=c_{268}=c_{358}=c_{369}=c_{45(12)}=c_{46(10)}=-1
       \newline
       c_{178}=c_{1(11)(12)}=2$&
$ S_1=Y_1+ 18\, x_5 x_6\partial_6\newline S_2=Y_2+3 x_6 \left( 3(x_1^2 - 4 x_3 x_5) + 4 E x_5 x_6 \right)\partial_6\newline S_3=Y_3+2 x_6 \Big[
-12 x_1 x_2 - 6 x_1^2 x_3 + 5 x_1^3 x_4 + 36 x_3^2 x_5 + 27 x_2 x_4 x_5 - 27 x_1 x_3 x_4 x_5 + 3 E (x_1^2 - 8 x_3 x_5 + 3 x_1 x_4 x_5) x_6
\Big]\partial_6\newline S_4=Y_4+3 x_6 \Big[
-9 x_2^2 + 66 x_1 x_2 x_3 - 21 x_1^2 x_3^2 - 30 x_1^2 x_2 x_4 + 10 x_1^3 x_3 x_4 - 48 x_3^3 x_5 - 108 x_2 x_3 x_4 x_5 + 108 x_1 x_3^2 x_4 x_5 + 54 x_1 x_2 x_4^2 x_5 - 54 x_1^2 x_3 x_4^2 x_5 + 6 x_1^3 x_4^3 x_5 + 6 E \big[
(8 x_3^2 + 3 x_2 x_4) x_5 - 3 x_1 (x_2 + 3 x_3 x_4 x_5) + x_1^2 (x_3 + 2 x_4^2 x_5)
\big] x_6
\Big]\partial_6$
       
       &\parbox{2.85cm}{\centering four-contact\\(on a dense subset)} \\
       \hline
       \end{tabular}}
        \caption{Finite-dimensional Lie algebras and Lie symmetries of Goursat distributions $\langle X_1,X_2\rangle$. We also comment which distributions are related to $k$-contact distributions. Goursat distributions on $\mathbb{R}^4$, $\mathbb{R}^5$, and $\mathbb{R}^6$ can be found in \cite{PR_01}. Existence of generating functions $A_{\mu1},\ldots,A_{\mu4}$ for $\mu=7,8$ giving rise to Reeb vector fields on $\mathbb{R}^6$ is still  an open problem. Note that  where $D=x_6+1$ 
 and  $E=x_5(x_1-x_5x_4)$.
 }
    \label{tab:Lie_systems}
\end{table}
\end{landscape}

\begin{theorem} Any Goursat distribution on \(\mathbb{R}^4\), \(\mathbb{R}^5\), or \(\mathbb{R}^6\) is locally equivalent to the Kumpera--Ruiz normal forms listed in Table~\ref{tab:Lie_systems}, which are spanned by the vector fields \(\langle X_1, X_2 \rangle\). These distributions are contained in VG Lie algebras and are related to \(k\)-contact distributions, as indicated in Table~\ref{tab:Lie_systems}.

\end{theorem}

\begin{proof} Showing that the Goursat distributions on $\mathbb{R}^4, \mathbb{R}^5$ and $\mathbb{R}^6$ are spanned by vector fields within VG Lie algebras while classes  one to five are related to $k$-contact distributions is immediate in terms of the data in Table \ref{tab:Lie_systems}. 

Let us show that class six is not related to a four-contact distribution. Note that $Y\in \mathfrak{X}(\mathbb{R}^6)$ is a Lie symmetry of $\mathcal{D}=\langle X_1,X_2\rangle$ with $X_1\wedge X_2\neq 0$ if and only if there exists $f\in C^\infty(\mathbb{R}^6)$  such that, with respect to the Schouten-Nijenhuis bracket \cite{Mar_97}, one has
$[Y,X_1\wedge X_2]=fX_1\wedge X_2.$

 The rank of $\mathcal{D}_{X_1,X_2}^4=\langle X_1,X_2,-X_3={\rm ad}_{X_2}(X_1),X_4={\rm ad}^2_{X_2}(X_1),-X_5={\rm ad}^3_{X_2}(X_1),X_6={\rm ad}^4_{X_2}(X_1)\rangle $ is five for $x_5=0$ and six on $x_5\neq 0$. Since the vector fields $X_1,X_3,X_4,X_5,X_6$ are invariant relative to $X_1$, it follows that $Y$ should leave the submanifold $x_5=0$ invariant \cite{LRS_24} and any four Lie symmetries should be tangent to the submanifold $x_5=0$, where ${\rm rk} \, \mathcal{D}_{X_1,X_2}^{4}=5$. Since $\mathcal{D}_{X_1,X_2}$ are tangent to that submanifold at $0$ too, Lie symmetries cannot span a supplementary to $\langle X_1,X_2\rangle$ at points with $x_5=0$, in particular, at $0$. Since every Goursat distribution of class six is locally diffeomorphic to a neighbourhood of zero in the representation in Table \ref{tab:Lie_systems}, it cannot be a four-contact distribution. 
 
To study classes 7 and 8,  long calculations performed with symbolic mathematical programs lead to the Lie symmetries $Y^A_7$ and $Y^A_8$, respectively, of the form
$$\textstyle
Y^A_\mu=  -\tfrac{\partial A}{\partial x_3}\partial_1 + \left(A - x_3 \tfrac{\partial A}{\partial x_3} \right)\partial_2+ \left(\tfrac{\partial A}{\partial x_1} + x_3 \tfrac{\partial A}{\partial x_2}\right)\partial_3+\sum_{\alpha=4}^6F^A_{\mu\alpha}\partial_\alpha,\quad \mu=7,8,
$$
parametrized by an arbitrary function $A=A(x_1,x_2,x_3)$ and where $F^A_{\mu 4},F^A_{\mu 5},F^A_{\mu 6}$ are univocally defined by $A$ and $\mu=7,8$. In particular, $A_{\mu i}=x_1^{4-i}(x_1x_3-3x_2)^{i-1}$ for $i=1,\ldots,4$  give a rise to a family of Reeb vector fields defined on a dense subset of $\mathbb{R}^6$ for classes 7,8  of the form $R_{\mu i}=Y_i+F_{\mu6}^i\partial_6$, where
\[
\begin{aligned}
Y_1 &= x_1^3\partial_2 + 3x_1^2\partial_3 + 6x_1\partial_4 - 6x_5^2\partial_5, \\[0.5em]
Y_2 &= -x_1^3\partial_1 - 3x_1^2 x_2\partial_2 - 6x_1 x_2\partial_3 
      - \left(6C - 3x_1^2 x_4\right)\partial_4 
      - 6x_5(x_1^2 - 2x_3 x_5)\partial_5, \\[0.5em]
Y_3 &= 2x_1^2 B\partial_1 
      + \left(9x_1 x_2^2 - x_1^3 x_3^2\right)\partial_2 
      + 3(3x_2 - x_1 x_3) C\partial_3 \\
    &\quad + \left(24x_2 x_3 - 6x_1 C x_4 + 2x_1^3 x_4^2\right)\partial_4 \\
    &\quad + 6x_5\left(x_1^2 x_3 - x_1^3 x_4 - (4x_3^2 + 3x_2 x_4)x_5 + 3x_1(x_2 + x_3 x_4 x_5)\right)\partial_5, \\[0.5em]
Y_4 &= -3x_1 B^2\partial_1 
      + (-3x_2 - 2x_1 x_3) B^2\partial_2 
      - 6x_3 B^2\partial_3 \\
    &\quad -3(3x_2 - x_1 x_3)\left(8x_3^2 - 9x_1 x_3 x_4 + x_4(3x_2 + 2x_1^2 x_4)\right)\partial_4 \\
    &\quad + 6x_5(2x_3 - x_1 x_4) \left(3x_1B + x_5(4x_3^2 - 7x_1 x_3 x_4 + x_4(9x_2 + x_1^2 x_4))\right)\partial_5,
\end{aligned}
\]
with $B=-3x_2+x_1x_3, \;C=x_2+x_1x_3$ and functions $F_{7 6}^1,\ldots,F_{76}^4,F_{8 6}^1\ldots,F_{86}^4$ are given in Table \ref{tab:Lie_systems}.
\end{proof}

\section{Applications to Lie systems}

There are many physical and mathematical problems related to Goursat distributions. Some of them are related to dynamical systems which can be, at least, locally described by the vector fields in Table \ref{tab:Lie_systems}. This relates them to Lie systems. Let us now focus on trailer systems and their relation to contact manifolds.

The \emph{$n$-trailer system} is the rank-two distribution defined on $\mathbb{R}^2\times (S^1)^{n+1}$ with local coordinates $\{\xi_{1},\xi_2,\theta_{0},\ldots,\theta_{n}\}$ given by \cite{PR_01} 
$$\textstyle
X_{1}=\partial_{\theta_{n}}, \quad X_{2}=\pi_{0}\cos(\theta_{0})\partial_{ \xi_{1}}+\pi_{0}\sin(\theta_{0})\partial_{ \xi_{2}}+\sum_{i=0}^{n-1}\pi_{i+1}\sin(\theta_{i+1}-\theta_{i})\partial_{ \theta_{i}},
$$ 
where $\pi_{i}=\prod_{j=i+1}^n\cos(\theta_{j}-\theta_{j-1})$ for $i=1,\ldots, n-1$ and $\pi_{n}=1$. Let us study the zero- and one-trailer systems. They are related to the analysis of certain dynamical systems \cite{NWS_01} and other mathematical problems.

The zero-trailer system is defined in $S^1\times \mathbb{R}^2$ and is associated with a unicycle-like mobile robot towing no trailers. Its dynamics is related to a $t$-dependent vector field $X=b_1(t)X_1+b_2(t)X_2$, where $X_1,X_2\in \mathfrak{X}(\mathbb{R}^2\times S^1)$ are included in a Lie algebra $V_0$ spanned by
$$
X_1=\partial_{ \theta_0},\quad  X_2=\cos(\theta_0)\partial_{ \xi_1}+\sin(\theta_0)\partial_{ \xi_2},\quad X_3=-\sin(\theta_0)\partial_{ \xi_1}+\cos(\theta_0)\partial_{ \xi_2}.
$$
Indeed, $[X_1,X_2]=X_3,[X_1,X_3]=-X_2$, and $[X_2,X_3]=0$. Then, $V_0$ is a three-dimensional Lie algebra isomorphic to $\mathfrak{iso}(2) = \mathfrak{so}(2) \subsetplus \mathbb{R}^{2}$, namely the Euclidean algebra in two dimensions.  Moreover, $X=\sum_{\alpha=1}^3b_\alpha(t)X_\alpha$ is a locally automorphic Lie system since $X_1\wedge X_2\wedge X_3$ is non-vanishing. Thus, the Lie algebra of vector fields commuting with the vector fields of $V_0$ is isomorphic to $\mathfrak{iso}(2)$ as well \cite{Gracia2019}. Note that $\langle X_1,X_2\rangle$ is a Goursat distribution on a manifold of dimension three. 

A computation shows that a basis of Lie symmetries of $V_0$ is given by 
$$
Y_{1} := - \xi_{2} \partial_{\xi_{1}} + \xi_{1} \partial_{\xi_{2}} + \partial_{\theta_{0}}, \qquad Y_2:=\partial_{ \xi_1}, \qquad Y_3:=\partial_{ \xi_2}.
$$

% $$
% M_1=\begin{bsmallmatrix}
% 0&-1&0\\
% 1&0&0\\
% 0&0&0
% \end{bsmallmatrix},\quad
% M_2=\begin{bsmallmatrix}
% 0&0&-1\\
% 0&0&0\\
% 0&0&0
% \end{bsmallmatrix},\quad
% M_3=\begin{bsmallmatrix}
% 0&0&0\\
% 0&0&-1\\
% 0&0&0
% \end{bsmallmatrix}
% $$
% Additionally, the Lie group action on itself can be generated by the elements of $\langle X_1,X_2,X_3\rangle$, namely for every $X \in \mathfrak{g}, m \in  \mathbb{R}^2 \times S^1$
% namely every element $g \in S^1\times \mathbb{R}^2$ can be obtained by the one-parameter group $\gamma_X:\mathbb{R}\rightarrow S^1\times \mathbb{R}^2$, given by $\gamma_X(t)=\exp(tX)$, where  

% \begin{align*}
% \exp(tX) \cdot m=&
% \begin{bsmallmatrix}
%  \cos (t)  & -\sin (t)& 1-\sin (t)-\cos (t) \\
%  \sin (t)  & \cos (t)  & -1+\cos (t)- \sin (t) \\
%  0 & 0 & 1 \\
% \end{bsmallmatrix}\begin{bsmallmatrix}
% \xi_1\\
% \xi_2\\
% \theta_0,
% \end{bsmallmatrix}
% \end{align*}
% % thus the Lie group action is given by
% $$
% \xi_1(t)=\left(\xi_1-\theta_0\right)\cos(\alpha t)-\left(\xi_2+\theta_0\right)\sin(\alpha t)+\theta_0,
% $$

% $$
% \xi_2(t)=\left(\xi_2+\theta_0\right)\cos(\alpha t)+\left(\xi_1-\theta_0\right)\sin(\alpha t)-\theta_0
% $$

The dual frame to $\{Y_{1}, Y_{2}, Y_{3} \}$ is 
$ \eta_{1} = \dd \theta_{0},  \eta_{2} = \dd \xi_{1} + \xi_{2} \dd \theta_{0},  \eta_{3} = \dd \xi_{2} - \xi_{1} \dd \theta_{0}.
$
Then, $\dd \eta_{3} = \eta_{1} \wedge \eta_{2}$ and $\eta_{3} \wedge \dd \eta_{3} \neq 0$. In an alternative manner, one can see that the distribution $\ker \eta_3=\langle Y_1,Y_2\rangle$ is maximally non-integrable because $[Y_1,Y_2]\wedge Y_1\wedge Y_2\neq 0$ and $X_1,X_2, X_3$ are  Lie symmetries of $\ker \eta_3$. Moreover, $Y_3$ is a Lie symmetry of $X$ for any $b_1(t),b_2(t),b_3(t)$. Hence, $\eta_{3}$ is a contact form whose associated Reeb vector field is $Y_{3}$. Then, $X$ is invariant relative to $Y_3$ and $X$ becomes a so-called {\it conservative contact Lie system }\cite{LR_23}. One can see that $\eta_2$ also satisfies that $\dd\eta_2=-\eta_1\wedge \eta_3$ and $\dd\eta_2\wedge \eta_2\neq 0$. Since $\eta_2$ is invariant relative to the vector fields $X_1,X_2,X_3$, i.e. $\mathcal{L}_{X_{i}}\eta_2=0$ for $i=1,2,3$, it follows that $X$ is also a conservative contact Lie system. It is indeed invariant relative to vector field $Y_2$, the Reeb vector field of $\eta_2$. One still can see that $X$ is invariant relative to the contact form $\eta_1+\eta_2$ and $X$ is invariant relative to $(Y_3+Y_2)/2$. 

The above shows that $X$ is invariant relative to the Lie group action $\phi:\mathrm{ISO}(2)\times \mathbb{R}^2\times S^1\rightarrow \mathbb{R}^2\times S^1$ obtained by integrating $Y_1,Y_2,Y_3$, namely 
\begin{equation}
\phi\left(\begin{bsmallmatrix}
 A_\theta &\vec{\lambda}\\
0&1\\
\end{bsmallmatrix},\begin{bsmallmatrix}
\vec{\xi}\\
\theta_0,
\end{bsmallmatrix}\right)=\begin{bsmallmatrix}
\vec{\lambda}+A_{\theta}\vec{\xi}\\
\theta_0+\theta
\end{bsmallmatrix},\qquad A_\theta=\begin{bsmallmatrix}
\cos\theta&-\sin\theta\\
\sin\theta&\cos\theta
\end{bsmallmatrix},
\end{equation}
and the general solution to $X$ can be written as
$
x(t)=\phi((A_\theta,{\bf \vec{\xi}}),x_p(t))$, for $\vec{\xi}\in \mathbb{R}^2$, and $\theta\in [0,2\pi[,
$
for a particular solution $x_p(t)$ of $X$. 
This shows how contact geometry can be used to obtain superposition rules for contact Lie systems whose Reeb vector fields are Lie symmetries of the system. 

% The one-trailer is also a Lie system related to a Goursat distribution on $\mathbb{R}^2\times (S^1)^2$ and a six-dimensional VG Lie algebra. 
% A car with front and rear wheels and the Martinet sphere are Lie systems \cite{Ram_02} related to Goursat distributions on $\mathbb{R}^4$. The front-wheel driven kinematic car pulling a trailer in  \cite{Ram_02} is related to a Goursat distribution on $\mathbb{R}^5$. Systems in chained forms are related to Goursat distributions on $\mathbb{R}^n$.
It is worth noting that one-trailer system is also a Lie system associated with a Goursat distribution on $\mathbb{R}^2 \times (S^1)^2$, which admits a six-dimensional Vessiot–Guldberg (VG) Lie algebra. By means of the classification of Goursat distributions in low-dimensional cases presented in Table \ref{tab:Lie_systems}, one can analyse the geometric structure and integrability properties of such system through $k$-contact geometry. It is noteworthy that several systems of interest in control theory are Lie systems related to Goursat distributions. For instance, the kinematic model of a car with both front and rear wheels, as well as the Martinet sphere, are known Lie systems \cite{Ram_02} that correspond to Goursat structures on $\mathbb{R}^4$. The front-wheel driven car pulling a trailer, as analysed in \cite{Ram_02}, leads to a Goursat distribution on $\mathbb{R}^5$, and falls under the classification framework introduced in Table \ref{tab:Lie_systems} of this work. 
Furthermore, systems presented in chained forms are naturally associated with Goursat distributions on $\mathbb{R}^n$. The presented approach demonstrates that these systems, at least in lower dimensions, can be systematically studied using the tools of $k$-contact geometry, which extend the classical contact structure paradigm.

\section{Cartan geometry and $k$-contact distributions}
Let us inspect examples of $k$-contact distributions motivated by Cartan geometry.
Let us start with a distribution  on $\mathbb{R}^6$ with coordinates $\{x_1,x_2,x_3,u_1,u_2,u_3\}$ given by
\[\mathcal{D}=\langle X_1:=\partial_{x_3}-x_2\partial_{u_1},\;X_2:=\partial_{x_1}-x_3\partial_{u_2},\;X_3:=\partial_{x_2}-x_1\partial_{u_3}\rangle.
\]
It is said to be 
a $(3,6)$ distribution since it has rank three and the vector fields taking values in it along with their Lie brackets span a distribution of rank six.  Noteworthy, $\mathcal{{D}}$ defines a flat parabolic (see \cite{CS09} for the definition of parabolic geometries) Cartan geometry of type $(\mathrm{SO}(3,4),P_1)$ and so it admits a Lie algebra of infinitesimal symmetries isomorphic to  $\mathfrak{so}(3,4)$. Three commuting Lie symmetries of $\mathcal{D}$ are given by $$
S_1=[X_1,X_3]=\partial_{u_1},\qquad  S_2=[X_2,X_1]=\partial_{u_2},\qquad S_3=[X_3,X_2]=\partial_{u_3}
$$
and $\mathcal{D}\oplus\langle S_1,S_2,S_3\rangle=T\mathbb{R}^6$. This makes $\mathcal{D}$ a three-contact structure.

Let us consider the classical \cite{Car10} example of a generic rank two distribution (so-called $(2,3,5)$ distribution) on $\mathbb{R}^5$ (a parabolic Cartan geometry of type $(\mathrm{G}_2,P_1)$) with local coordinates $\{x,y,z,p,q\}$ given by (see \cite{Ran_16} for details)
$$\mathcal{D}=\langle \partial_q,\partial_x+p\partial_y+q\partial_p+F(q)\partial_z\rangle,$$
for a function $F\in C^\infty(\mathbb{R}^5)$ satisfying $\partial_q^2F\neq 0$. Since the Lie brackets in $\mathcal{D}$ with themselves span a distribution of rank three, $\mathcal{D}$ is maximally non-integrable. Moreover, it admits three commuting Lie symmetries $S_1=\partial_x,S_2=\partial_y,S_3=\partial_z$ and $\mathcal{D}\;\oplus\;\langle S_1,S_2,S_3\rangle=T\mathbb{R}^5$, so $\mathcal{D}$ is a three-contact distribution.  

Finally,  let us study a $(4,7)$ distribution belonging to the so-called \emph{quaternionic contact} (qc) \emph{structures} in dimension $7$ \cite{MS18,Nur25}. In \cite{MS18}, it is proved that 7-dimensional qc structures are Cartan geometries of type $(\mathrm{Sp}(4,1),P)$, where $P$ is the parabolic subgroup stabilizing a certain complex 2-plane in $\mathrm{Sp}(4,1)$. Consider coordinates $\{u_1,u_2,u_3,x_1,x_2,x_3,x_4\}$ on $\mathbb{R}^7$. Following \cite{Nur25}, one may define an $\mathfrak{sp}(4,1)$-symmetric (Cartan-flat) qc structure by a rank 4 distribution 
        $$\mathcal{D}=\langle x_2\partial_{u_1}-x_1\partial_{u_2}+x_3\partial_{u_3}+\partial_{x_4},x_1\partial_{u_1}-x_2\partial_{u_2}-\partial_{x_3}, x_1\partial_{u_3}+\partial_{x_2}, \partial_{x_1}\rangle.$$
        Then, $S_1=\partial_{u_1},S_2=\partial_{u_2},S_3=\partial_{u_3}$ are  commuting Lie symmetries of $\mathcal{D}$ and $\mathcal{D}\oplus\langle S_1,S_2,S_3\rangle=T\mathbb{R}^7$. Thus, flat qc structures define three-contact structures.

%\section{Conclusions}

\begin{credits}
\subsubsection{\ackname} Fruitful discussions with O. Carballal, X. Rivas and J. de Lucas are acknowledged. Specially, we thank you very much Ian Anderson for the two talks given for the Gamma Seminar of the University of Warsaw and additional information, which pointed us how the software package \verb|DifferentialGeometry| could be used to complete the Reeb vector fields of classes 7 and 8 in Table \ref{tab:Lie_systems}.  We thank the referees for useful comments that improved the quality of our work.

\subsubsection{\discintname}
The authors have no competing interests to declare that are relevant to the content of this article.
% Or: Author A has received research
% grants from Company W. Author B has received a speaker honorarium from
% Company X and owns stock in Company Y. Author C is a member of committee Z.
\end{credits}

\bibliographystyle{splncs04}
\bibliography{references}

\begin{thebibliography}{10}
\providecommand{\url}[1]{\texttt{#1}}
\providecommand{\urlprefix}{URL }
\providecommand{\doi}[1]{https://doi.org/#1}

\bibitem{BBHLS_15}
Ballesteros, A., Blasco, A., Herranz, F., de~Lucas, J., Sard{\'{o}}n, C.: {Lie-Hamilton systems on the plane: Properties, classification and applications}. Journal of Differential Equations  \textbf{258}(8) (2015), \href{https://doi.org/10.1016/j.jde.2014.12.031}{10.1016/j.jde.2014.12.031}

\bibitem{CS09}
{\v{C}}ap, A., Slovák, J.: Parabolic geometries I. Background and general theory, Math. Surv. and Mono., vol.~154. Amer. Math. Soc., Providence (2009), \href{https://doi.org/10.1090/surv/154}{10.1090/surv/154}

\bibitem{CGM_00}
Cariñena, J.F., Grabowski, J., Marmo, G.: {Lie--Scheffers systems: a geometric approach}. Napoli Series on Physics and Astrophysics, Bibliopolis, Naples (2000)

\bibitem{CL_11}
Cariñena, J.F., de~Lucas, J.: {Lie systems: theory, generalisations, and applications}. Dissertationes Math.  \textbf{{\bf 479}},  1--162 (2011), \href{https://doi.org/10.4064/dm479-0-1}{10.4064/dm479-0-1}

\bibitem{Car10}
Cartan, E.: Les syst\`emes de {P}faff, \`a cinq variables et les \'{e}quations aux d\'{e}riv\'{e}es partielles du second ordre. Ann. Sci. \'{E}cole Norm. Sup. (3)  \textbf{27},  109--192 (1910), \href{http://www.numdam.org/item?id=ASENS_1910_3_27__109_0}{10.24033/asens.618}

\bibitem{GGMRR_20}
Gaset, J., Gràcia, X., Muñoz-Lecanda, M.C., Rivas, X., Román-Roy, N.: {A contact geometry framework for field theories with dissipation}. Ann. Phys.  \textbf{{\bf 414}},  168092 (2020), \href{https://doi.org/10.1016/j.aop.2020.168092}{10.1016/j.aop.2020.168092}

\bibitem{GGMRR_21}
Gaset, J., Gràcia, X., Muñoz-Lecanda, M.C., Rivas, X., Román-Roy, N.: {A $k$-contact Lagrangian formulation for nonconservative field theories}. Rep. Math. Phys.  \textbf{{\bf 87}}(3),  347--368 (2021), \href{https://doi.org/10.1016/S0034-4877(21)00041-0}{10.1016/S0034-4877(21)00041-0}

\bibitem{Gracia2019}
Gracia, X., De~Lucas, J., Munoz-Lecanda, M.C., Vilarino, S.: {Multisymplectic structures and invariant tensors for Lie systems}. Journal of Physics A: Mathematical and Theoretical  \textbf{52}(21) (4 2019), \href{https://doi.org/10.1088/1751-8121/ab15f2}{10.1088/1751-8121/ab15f2}

\bibitem{KR82}
Kumpera, A., Ruiz, C.: Sur l'\'{e}quivalence locale des syst\`emes de {P}faff en drapeau. In: Monge-{A}mp\`ere equations and related topics ({F}lorence, 1980), pp. 201--248. Ist. Naz. Alta Mat. Francesco Severi, Rome (1982)

\bibitem{Lie1}
Lie, S., Scheffers, G.: Vorlesungen über continuierliche Gruppen mit geometrischen und anderen Anwendungen. Leipzig, B.G. Teubner (1893), \href{https://doi.org/10.5962/bhl.title.18549}{10.5962/bhl.title.18549}

\bibitem{LR_23}
de~Lucas, J., Rivas, X.: {Contact Lie systems: theory and applications}. J. Phys. A: Math. Theor.  \textbf{{\bf 56}}(33),  335203 (2023), \href{https://doi.org/10.1088/1751-8121/ace0e7}{10.1088/1751-8121/ace0e7}

\bibitem{LRS_24}
de~Lucas, J., Rivas, X., Sobczak, T.: {Foundations on $k$-contact geometry} (2024), \href{https://doi.org/10.48550/arXiv.2409.11001}{10.48550/arXiv.2409.11001}

\bibitem{LS_20}
de~Lucas, J., Sardón, C.: {A guide to Lie systems with compatible geometric structures}. World Scientific Publishing Co. Pte. Ltd., Singapore (2020), \href{https://doi.org/10.1142/Q0208}{10.1142/Q0208}

\bibitem{Mar_97}
Marle, C.M.: {The Schouten--Nijenhuis bracket and interior products}. J. Geom. Phys.  \textbf{{\bf 23}}(3--4),  350--359 (1997), \href{https://doi.org/10.1016/S0393-0440(97)80009-5}{10.1016/S0393-0440(97)80009-5}

\bibitem{MS18}
Minchev, I., Slov\'{a}k, J.: On the equivalence of quaternionic contact structures. Ann. Global Anal. Geom.  \textbf{53}(3),  331--375 (2018), \href{https://doi.org/10.1007/s10455-017-9580-2}{10.1007/s10455-017-9580-2}

\bibitem{NWS_01}
Nakamura, Y., Chung, W., Sordalen, O.: Design and control of the nonholonomic manipulator. IEEE Transactions on Robotics and Automation  \textbf{17}(1),  48--59 (2001), \href{https://doi.org/10.1109/70.917082}{10.1109/70.917082}

\bibitem{Nur25}
Nurowski, P.: Exceptional real lie algebras $\mathfrak{f}_4$ and $\mathfrak{e}_6$ via contactifications. J. Inst. Math. Jussieu  \textbf{24}(1),  157–201 (2025), \href{https://doi.org/10.1017/S1474748024000173}{10.1017/S1474748024000173}

\bibitem{PR_01}
Pasillas-Lépine, W., Respondek, W.: {On the Geometry of Goursat Structures}. ESAIM - Control Optim. Calc. Var.  \textbf{{\bf 6}},  119--181 (2001), \href{https://doi.org/10.1051/cocv:2001106}{10.1051/cocv:2001106}

\bibitem{Ram_02}
Ramos, A.: Sistemas de Lie y sus aplicaciones en Física y Teoría de Control. Ph.D. thesis, Universidad de Zaragoza (2002), \href{https://doi.org/10.48550/arXiv.1106.3775}{10.48550/arXiv.1106.3775}

\bibitem{Ran_16}
Randall, M.: {Flat (2,3,5)-Distributions and Chazy's Equations}. SIGMA. Symmetry, Integrability and Geometry: Methods and Applications  \textbf{12}, ~029 (3 2016), \href{https://doi.org/10.3842/SIGMA.2016.029}{10.3842/SIGMA.2016.029}

\end{thebibliography}

\end{document}